\documentclass[reqno]{amsart}
\usepackage{amssymb,amsthm,amsmath}

\usepackage{mathptmx}       
\usepackage{helvet}         
\usepackage{courier}        
\usepackage{graphicx}
\usepackage{eucal}
\usepackage[all]{xy}
\usepackage[]{float}
\usepackage[]{epsfig}
\usepackage[]{pstricks}
\usepackage{enumerate}

\usepackage[active]{srcltx}
\usepackage{flafter}
\usepackage{amscd,amsthm}
\usepackage{amsfonts}
\usepackage{graphicx}
\usepackage{verbatim}
\usepackage[ansinew]{inputenc}
\makeatletter \@addtoreset{equation}{section} \makeatother

\renewcommand\thetable{\thesection.\@arabic\c@table}
\newtheorem{theorem}{Theorem}[section]
\newtheorem{lemma}[theorem]{Lemma}
\newtheorem{proposition}[theorem]{Proposition}

\theoremstyle{definition}

\newtheorem{definition}[theorem]{Definition}
\newtheorem{ejemplo}[theorem]{Example}
\newtheorem{remark}[theorem]{Remark}

   \usepackage{wasysym}

\newcommand{\eps}{\varepsilon}

\newcommand{\LLp}{L^2(\bb T)}
\newcommand{\Lp}{L^1(\bb T)}

\newcommand{\HH}{H^2([0,1])}
\newcommand{\Hp}{H^1(\mathbb T)}
\newcommand{\Ha}{{H^1_a}}
\newcommand{\WW}{W^{2,1}([0,1])}
\newcommand{\WWp}{W^{2,1}(\mathbb T)}
\newcommand{\HWp}{H^1_W(\mathbb T)}
\newcommand{\HHp}{H^2(\mathbb T)}
\newcommand{\COp}{C(\mathbb T)}
\newcommand{\CCp}{C^2(\mathbb T)}
\newcommand{\DW}{\mathcal D_W(\bb T)}

\newcommand{\LW}{\mathcal L_W}
\newcommand{\A}{\mathcal A}
\renewcommand{\Re}{\operatorname{Re}}

\newcommand{\mc}[1]{{\mathcal #1}}

\newcommand{\bb}[1]{{\mathbb #1}}

\newcommand{\p}{\partial}
\newcommand{\pfrac}[2]{\genfrac{}{}{}{1}{#1}{#2}}
\newcommand{\ppfrac}[2]{\genfrac{}{}{}{2}{#1}{#2}}

\newcommand{\dd}{\displaystyle}

\newcommand{\R}{\mathbb R}
\newcommand{\C}{\mathbb C}
\newcommand{\Z}{\mathbb Z}

\newcommand{\F}{\mathcal F} 

\renewcommand{\o}{\overline}


\date{}

\begin{document}

\title[Continuous dependence on the derivative]{Continuous dependence on the derivative \\ of generalized heat equations }

\author{Tertuliano Franco}
\address{UFBA\\
 Instituto de Matem\'atica, Campus de Ondina, Av. Adhemar de Barros, S/N. CEP 40170-110\\
Salvador, Brasil}
\curraddr{}
\email{tertu@impa.br}

\author{Juli\'an Haddad}
\address{UFBA\\
 Instituto de Matem\'atica, Campus de Ondina, Av. Adhemar de Barros, S/N. CEP 40170-110\\
Salvador, Brasil}
\curraddr{}
\email{jhaddad@dm.uba.ar}

\begin{abstract}
We consider here a generalized heat equation $\p_t \rho=\pfrac{d}{dx}\pfrac{d}{dW}\rho$, where $W$ is a finite measure on the one dimensional torus, and $\pfrac{d}{dW}$ is the Radon-Nikodym derivative with respect to $W$.
  Such equation has appeared in different contexts, being related to physical systems and representing a large class of classical and non-classical parabolic equations.
  As a natural assumption on $W$, we require that the Lebesgue measure is absolutely continuous with respect to $W$.     
    The main result here presented consists in proving, for a suitable topology, a continuous dependence of the solution $\rho$ as a function of $W$.

\end{abstract}

\maketitle

\section{Introduction}\label{s1} 
The subject of partial differential equations related to generalized derivatives is a somewhat recent and unexploited research theme, with connections with Physics and Probability.
By a generalized derivative we mean, \emph{grosso modo}, a Radon-Nikodym derivative. In this paper, we are concerned with the following partial differential equation
\begin{equation}\label{W_PDE}
\begin{cases}
 \p_t \rho=\pfrac{d}{dx}\pfrac{d}{dW}\rho & \textrm{in \;} (0,\infty)\times \bb T\\
\rho(\cdot, 0) = h(\cdot)& \hbox{in \;} \bb T,
\end{cases}
\end{equation}
where $\bb T = \bb R / \bb Z$ is the one dimensional torus and $W:\R \to \R$  is a right continuous and periodically increasing function in the sense that $W(x+1) - W(x) = 1$ for every $x \in \R$. Or else, $W$ can be understood as the distribution function of a probability measure $\mu$ on the torus $\bb T$.

A function $f$ for which $\frac {df}{dW}$ is well defined and differentiable may have jump discontinuities at the discontinuity points of $W$, namely the points with positive $\mu$-measure.
In the Section 2.7 of \cite{fl}, it was proven the existence of a unique weak solution belonging to the  space $L^2([0,T], \HWp)$ for the equation \eqref{W_PDE}, where $\HWp$ is a suitable Sobolev-type space which admits discontinuous functions.

We restrict ourselves to the class of measures $\mu$ for which the Lebesgue measure is absolutely continuous with respect to $\mu$.
Our main result is a continuous dependence of the unique solution $\rho$ of \eqref{W_PDE} with respect to $W$.
By means of a sequence of transformations, we not only solve the problem of continuity but we actually give an explicit construction of the weak solution.

The subject of dynamics related to generalized derivatives has connections with different areas.  For instance,  the book \cite{m} studies one-dimensional Markov processes whose generators involve Radon-Nikodym derivatives. In fractal analysis we cite \cite{U}. 
Related to  Krein-Fellers operators, see \cite{f,lo1,lo2}.
Partial differential equations related to this operator $\pfrac{d}{dx}\pfrac{d}{dW}$ naturally come out in hydrodynamic limit and fluctuations of interacting particle systems in non-homogeneous medium, see \cite{fjl, fsv, fl}.

The equation \eqref{W_PDE} is, for some cases of the measure $\mu$, in correspondence with  classical
PDE's. Clearly, if $\mu$ is the Lebesgue measure, the PDE \eqref{W_PDE} is  equivalent to the classical heat equation in the one dimensional torus.  

Additionally, in the case where $\mu=\mc L+b\delta_0$, where $\mc L$  is the Lebesgue measure, and $\delta_0$ is the Dirac delta measure at zero, the PDE \eqref{W_PDE} is equivalent to the following heat equation with Robin's boundary conditions:
\begin{equation}\label{Robin}
\begin{cases}
 \p_t \rho(t,x) =\p_{xx}\rho(t,x), & \textrm{ for } t>0, x\in (0,1)\\
 b\big(\rho(t,1)-\rho(t,0)\big)=\p_x \rho(t,0)=\p_x \rho(t,1),& \textrm{ for } t>0,\\
\rho(0,x) = h(x), & \textrm{ for } x\in [0,1]. 
\end{cases}
\end{equation}
This equivalence was showed in \cite{fgn2}. Notice that the boundary conditions above represent the Fourier's Law: the rate of heat transfer across the interface between two media is proportional to the difference of temperature in each medium. 
In this case, the rate is given by the partial derivatives $\p_x \rho(t,0)=\p_x \rho(t,1)$ and the difference of temperature is given by $\rho(t,1)-\rho(t,0)$.

In \cite{fgn2}, it was also described the behaviour of the solution $\rho=\rho^b$ of the equation above as a function of the parameter $b$. It is proven that when $b\to \infty$, the function $\rho^b$ converges to the solution of the heat equation with Newmann's boundary conditions. When $b\to 0$, the convergence is towards the solution of the heat equation with periodic boundary conditions.
Our main theorem covers this last case with much more generality.

 The outline of the paper is the following. 
 In Section \ref{s2} we present heuristics on how one can deduce an equivalent equation for the PDE \eqref{W_PDE} and the proof's scheme about continuous dependence of solutions with respect to $W$.  
 In Section \ref{s3} precise definitions and statements are given.  
 In Section \ref{s4} we transform the equation \eqref{W_PDE} into a classical PDE with continuous weak solutions.
 In Section \ref{s5} we deal with the continuous dependence of the equivalent version of \eqref{W_PDE} by means of a careful analysis on its Fourier transform.
Some auxiliary results are left to the Appendix.

\section{Some interpretations and proof's scheme}\label{s2}
In this section we informally discuss the subject of this paper and the proof's general idea.
All arguments ahead are of heuristic nature. Precise definitions and statements will be presented in the next section.

Is well known that the heat equation may be derived from Fourier's Law, which states that the heat transfer $q$  is proportional to the negative gradient of temperature $\rho$, or else,
\begin{equation}
\label{fourierlaw}
q = -k \rho_x\;.
\end{equation}
An argument about conservation of energy leads to
\begin{equation}\label{capacity}
c \rho_t = -q_x = (k \rho_x)_x\;,
\end{equation}
where $c$ and $k$ are functions of the position $x$.
In physical nomenclature, $c$ is the heat capacity and $k$ is the thermal conductivity.
Equation \eqref{W_PDE} models the case where $k$ may be degenerate in the following sense:
assuming for a moment that $W$ is differentiable, equation \eqref{W_PDE} takes the form
\[
	\rho_t = (\ppfrac 1 {W'} \rho_x)_x\;,
\]
so $W'$ is the inverse of $k$ and represents the thermal resistance.
Fourier's Law \eqref{fourierlaw} takes henceforth the form 
\begin{equation}\label{ratio}
W' q = -\rho_x\;.
\end{equation}
When $W$ is not differentiable (possibly not even continuous), we shall interpret this equality with $W'$ and $\rho_x$ as Schwartz distributions, and $q$ a continuous function.
Let us consider the case when $W'(x) = 1 + \delta_{1/2}(x)$, hence $W$ and $\rho$ must have both a jump discontinuity at $x=\pfrac 12$.
From \eqref{ratio}, we see that the ratio between the size of the jumps is the heat transfer at that point.
This agrees with the Robin's boundary conditions in \eqref{Robin} mentioned in the introduction.

Keeping this interpretation in mind, we will reparametrize the interval $[0,1]$ in such a way that the thermal conductivity becomes constant, leading to a PDE with the classical Laplacian operator (with possibly zero heat capacity at some points). Roughly speaking, we are going to ``stretch" the support of the singular part of $\mu$ with respect to the  Lebesgue measure.  For instance, as showed in the Figure \ref{fig1}, the point $\pfrac12$ in the left graphic is transformed into the interval $[1,2]$ in the right graphic.

Being $W$ a strictly increasing function, it has a continuous left inverse $w:[0,1] \to [0,1]$ such that $w(W(y)) = y$ for all $y \in [0,1]$.
As we shall see, a function $h$ with discontinuities at the same points as $W$  may be represented as a composition $h(y) = f(W(y))$ for some continuous function $f$.
Under the change of variables $y = w(x)$, equation \eqref{W_PDE} becomes
\begin{equation}
	\label{a_equation}
	\left\{
		\begin{array}{l}
		a(x) v_t(t,x) = v_{xx}(t,x)\\
		a(x) v(0,x) = a(x) f(x)
		\end{array}
	\right.
\end{equation}
where, according to \eqref{capacity}, the function $a = w'$ plays the role of the heat capacity. The general strategy will be to establish the continuity of the solution $v$ of \eqref{a_equation} with respect to $a$ in a convenient function space, and then to prove that the composition $\rho(t,y) = v(t, W(y))$ is the solution of \eqref{W_PDE}.
This transformation puts together in the same space functions which are discontinuous in distinct sets.
As usual, the topology on the measures will be given by the vague convergence.

\begin{figure}[ht]
\centering
\includegraphics[scale=.75]{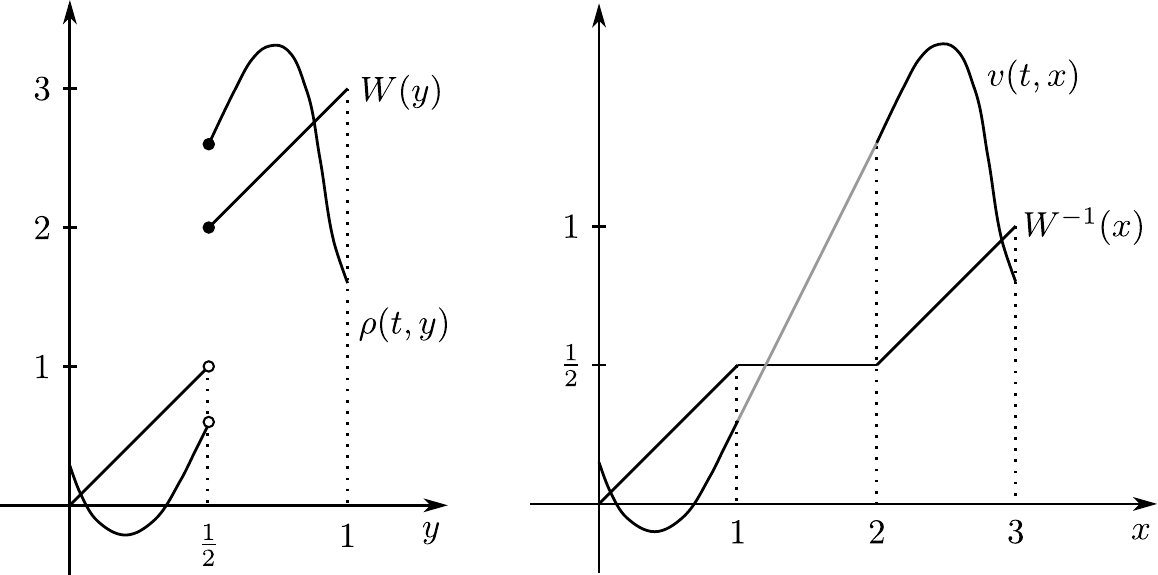} 
\caption{Transformation between equations \eqref{W_PDE} and \eqref{a_equation}. The functions $\rho$ and $v$ are related by $\rho(t,y) = v(t,W(y))$. The grey line is a $C^1$-linear interpolation.}
\label{fig1}
\end{figure}

Observe that if $W$ has a jump discontinuity at a point $x_0$, namely $W(x_0^-) = r < s = W(x_0^+)$, then $a = 0$ in the interval $[r,s]$.
This corresponds to an interval with zero heat capacity and since \eqref{a_equation} reduces to $v_{xx} = 0$, the temperature must be the linear interpolation of the values of $v$ at the end points of the interval, see Figure \ref{fig1}.
This fact may be interpreted as an ``infinite dispersion'' phenomena: any initial temperature at $[r,s]$ is completely dispersed at any positive time, and the initial condition $v(0,x) = f(x)$ will be satisfied only when $a(x) \neq 0$.
This justifies the second equation in \eqref{a_equation}.

In principle, the function $v$ is defined only for $t \in [0, \infty)$.
Extending $v$ as zero in the negative half line, we can take the Fourier transform with respect to time in \eqref{a_equation}, obtaining
\[
	a(x)(i \xi \hat v (\xi, x) - f(x)) = \hat v_{xx}(\xi, x).
\]
The term $f(x)$ appears because $v$ is discontinuous at $t=0$ so $v_t$ has a Dirac delta.
Now this equation is uncoupled in $\xi$ so it may be viewed as a one parameter family of periodic complex ODE's
\begin{equation}
	\label{xi_ODE}
	\left\{
	\begin{array}{l}
		-u''(x) + i \xi a(x) u(x) = a(x) f(x),\\
		u(x+1) = u(x),\\
	\end{array}
	\right.
\end{equation}
with $\xi\in \bb R$ as the parameter.

At this point, the classical theory of ODE's assures the existence of a unique solution that depends continuously on $a$ and $\xi$.
The main difficulty here is to show that we are able to anti-transform $u$ with respect to $\xi$ without loosing continuity. This is the subject of the Section \ref{s5}.

\section{Definitions and statements}\label{s3}

To a probability measure $\mu$ on the torus we can associate a unique right-continuous function $W:\R \to \R$ such that $W(0) = 0$, $W(x+1) - W(x) = 1$ for every $x \in \R$ and such that if $(a,b]$ represents an interval in the torus then
\[\mu((a,b]) = W(b) - W(a).\]
The function $W$ completely characterizes the measure $\mu$.

\begin{definition}
\label{defW}
We denote by $\mc W$ the set of functions $W:\R \to \R$ as above 
associated to the probability measures $\mu$ on the torus $\bb T$ such that the Lebesgue measure is absolutely continuous with respect to $\mu$.
\end{definition}

It follows that $W$ is strictly increasing.
This restriction for the set $\mc W$ is stronger than the one assumed in \cite{fl} namely, that $\mu(I) > 0$ for every open interval $I \subseteq \bb T$.
See Remark \ref{weaker_condition} in the appendix.

\begin{definition}
	We say that a sequence $W_n \in \mc W$ converges vaguely to $W \in \mc W$ if
	\[\int \phi d W_n \to \int \phi d W\]
	for every function $\phi \in \COp$.
\end{definition}
It is well known that this convergence is equivalent to the pointwise convergence of $W_n$ to $W$ in the continuity points of $W$, see for instance \cite{R}.

We use the notation $\langle \cdot, \cdot \rangle_{L^2}$ for the usual inner product in $\LLp$. Also $\HHp, \WWp$ will stand for the usual periodic-Sobolev spaces and $C^\alpha(\bb T)$ for the space of periodic $\alpha$-H\"older continuous functions.
Now we present the definitions concerning the generalized derivative as in \cite{fjl} and \cite{fl}.

\subsection{The generalized derivative}

For a function $f:\bb T \to\bb R$, we define $\frac{d}{dW}$ as follows:
\begin{equation*}
\frac{d f}{dW} (x) = \lim_{\eps\rightarrow 0} \frac{f(x+\eps)
-f(x)}{W(x+\eps) -W(x)},
\end{equation*}
if the above limit exists and is finite. 

\begin{definition}
\label{defDW}
Denote by $\DW$ the set of functions $f$ such that

\begin{equation}
\label{f17}
f(x) \;=\; b \;+\; c W(x)\; +\; \int_{(0,x]} \int_0^y g(z)
\, dz \; dW(y)
\end{equation}
for some function $g$ in $\LLp$ and some $b,c\in \bb R$, with
\begin{equation}
\label{f14}
c W(1) \;+\; \int_{\bb T} \int_0^y g(z) \, dz \; dW(y)  \;=\;0\;, \quad
\int_{\bb T} g(z) \, dz \;=\;0 \;.
\end{equation}
One can check that the function $g$, as well as the constants $b$,$c$, are unique.
The first requirement corresponds to the boundary condition
$f(1)=f(0)$ and the second one to the boundary condition $(df/dW) (1)
= (df/dW) (0)$.

Define the operator $\LW : \DW \to \LLp$ by
\begin{equation*}
\LW f \;=\; \frac{d}{dx} \frac{d}{dW}f \;=\;
\frac{d}{dx}\left(\frac{df}{dW}\right).
\end{equation*}
It is easy to see that $\LW f = g$ a.e in the notation of \eqref{f17}.

\end{definition}

\begin{definition}
\label{weak_W_PDE}
We say that a measurable bounded function $\rho: \R_+ \times \bb T \to \R$ is a weak solution of \eqref{W_PDE} if for all functions $\psi \in \DW$ and every $T>0$,
\[
	\langle \rho(T), \psi \rangle_{L^2} - \langle h, \psi \rangle_{L^2} = \int_0^T \langle \rho(t), \LW \psi \rangle_{L^2} dt.
\]
Here $\rho(T)$ denotes the function $\rho(T, \cdot)$.
\end{definition}

\subsection{Statements}
We are in position to state our main results:

\begin{proposition}
	\label{DW_parametrization_intro}
	Let $h \in \DW$. Then there exists $f \in \HHp$ such that $h(y) = f(W(y))$.
\end{proposition}

\begin{theorem}\label{mainthm} 
Let $W_n$, $W \in \mc W$.
Let $\rho_n$, $\rho$ be the corresponding unique weak solutions of the equations
\begin{equation*}
\begin{cases}
	& \p_t \rho=\pfrac{d}{dx}\pfrac{d}{d{W_n}}\rho\\
& \rho(\cdot, 0) = h_n\\
\end{cases}
\end{equation*}
and \eqref{W_PDE} respectively,
where $h_n \in \mc D_{W_n}$ and $h \in \DW$.
Note that by Proposition \ref{DW_parametrization_intro}, $h_n = f_n(W_n)$ and $h = f(W)$ for some $f_n, f \in \HHp$.
Assume that 
\begin{enumerate}[(i)]
\item $W_n \to W$ vaguely;

\item \label{compatibility} the functions $\mc L_{W_n}h_n$ are uniformly bounded;
\item $f_n \to f$ in $\HHp$.
\end{enumerate}

Then there exist continuous functions $v_n, v: \R_+ \times \bb T \to \R$ such that 
\begin{itemize}

\item the functions $v_n, v$ are the unique weak solutions of equation \eqref{a_equation} with initial conditions $f_n$ and $f$, respectively;

\item $\rho_n(t,x) = v_n(t,W_n(x))$ and $\rho(t,x) = v(t,W(x))$;

\item for every $\eps > 0$, the functions $v_n, v \in C(\bb T,C^{\frac 12 - \eps}(\R_+))$ and $v_n \to v$ in the topology of $C(\bb T,C^{\frac 12 - \eps}(\R_+))$.
\end{itemize}

\end{theorem}

In order to illustrate the range of applicability of our theorem, we present some examples.
Denote by $\mc L$ the Lebesgue measure on $\bb T$ and by ${\bf 1}_A$ the indicator function of a set $A$.

\begin{ejemplo}
\label{lebesgueplusdelta}
Consider the measure $\frac 12 \mc L + \frac 12 \delta_{\frac 12}$.
In this case, the function $a$ from equation \eqref{a_equation} is given by $a = {\bf 1}_{[0, \frac 14]} + {\bf 1}_{[\frac 34,1]}$.
We observe that equation \eqref{a_equation} is the classical (periodic) heat equation for $x \in [0, \frac 14] \cup [\frac 34, 1]$ while for $x \in [\frac 14, \frac 34]$ it reduces to $v_{xx} = 0$.
As a consequence, if we assume that $v_x$ is continuous then
\[v_x(t,\pfrac 14) = v_x(t, \pfrac 34) = 2 \big( v(t, \pfrac 34) - v(t, \pfrac 14)\big),\] for all $t > 0$.
The function $\rho(t,x) = v(t, W(x))$ satisfies the Robin's boundary conditions as in \eqref{Robin}, with $b=2$.

If instead we consider $(1-c) \mc L + c \delta_p$ then the Theorem \ref{mainthm} guarantees that the solution varies continuously with respect to $p \in \bb T$ and $c \in [0,1]$.

\end{ejemplo}

\begin{ejemplo}
The Laplacian operator with respect to fractal measures was considered in \cite{U} where its properties of self-similarities are exploited.
Consider $W(x) = \frac 12 x + \frac 12 \mc C(x)$ where $\mc C$ is the usual ternary Cantor ``staircase'' function.
Observe that if $C \subset [0,1]$ is the Cantor set, then $W(C)$ is a cantor-like set of positive measure and $a = {\bf 1}_{W(C)}$.

Consider the usual uniform approximation $\mc C_n \to \mc C$ by piecewise-linear continuous functions, and $W_n(x) = \frac 12 x + \frac 12 \mc C_n(x)$. Theorem \ref{mainthm} is applicable to this situation, although the corresponding solutions of \eqref{W_PDE} and \eqref{a_equation} are hard to describe.
\end{ejemplo}

\begin{ejemplo}
\label{piquito}

Consider the measures $\mu_n = \frac 12 \mc L + \frac 14 \delta_{(\frac 12 -\frac 1n)} + \frac 14 \delta_{(\frac 12 +\frac 1n)}$ whose vague limit is $\mu = \frac 12 \mc L + \frac 12 \delta_{\frac 12}$.
As in Example \ref{lebesgueplusdelta}, the solutions exhibit Robin's boundary conditions at the points $\frac 12 \pm \frac 1n$ and they converge to the solution of Example \ref{lebesgueplusdelta}.
In words, the two boundary conditions overlap in the limit.
\end{ejemplo}
\begin{ejemplo}
\label{piquitomalo}
Regarding equation \eqref{a_equation} in the situation of the previous example, the functions $a_n$ are shown in 
 \ref{fig2}.
Assume that the initial condition $f$ is also as in Figure \ref{fig2}.
As described in the introduction, for fixed $t > 0$ the solution of \eqref{a_equation} is linear in the intervals where $a = 0$, therefore the solution $\rho_n(t, \frac 12)$ do not converge to $\rho(t, \frac 12)$ uniformly in $t$, and the theorem fails.
In this case the convergence is only $L^2$ in time.

This counterexample is not relevant to equation \eqref{W_PDE} since the initial conditions $f(W_n)$ do not converge at $x = \frac 12$, but it suggests there should exist a compatibility condition between $a$ and $f$ in order to have the desired continuity.
This condition is contained in the requirement of Theorem \ref{mainthm} that $f(W_n) \in \DW$ which does not hold in this case.
The reader should compare this example with Proposition \ref{DW_parametrization} in the next section.
\end{ejemplo}

\begin{figure}[ht]
\centering
\includegraphics[scale=0.7]{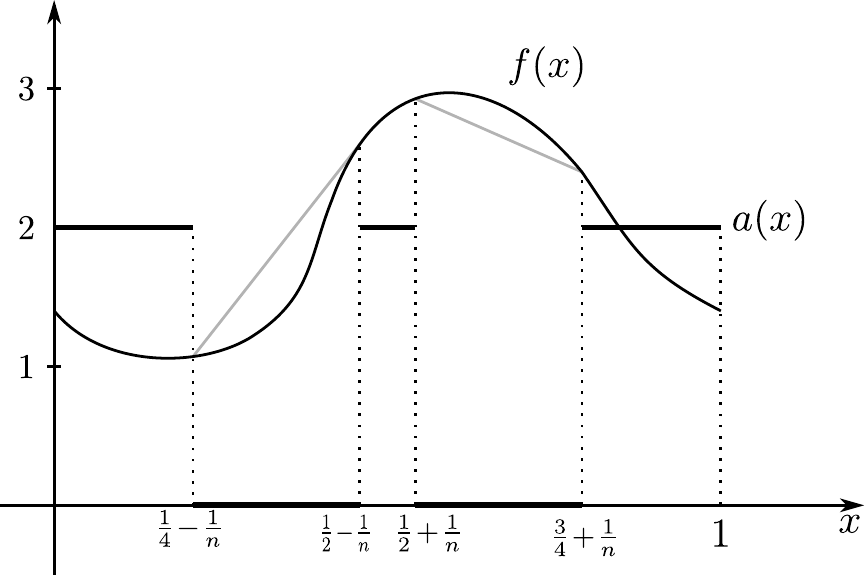} 
\caption{Initial compatibility. The grey segments are linear interpolations where the function $a$ vanishes.}
\label{fig2}
\end{figure}


\section{An equivalent version for the partial equation \eqref{W_PDE}}\label{s4}

A strictly increasing (not necessarily continuous) function $W:[0,1] \to [0,1]$ has a generalized inverse $w:[0,1] \to [0,1]$ defined as 
\begin{equation*}
w(s):=\sup\{r\,;\,W(r)\leq s\}\,.
\end{equation*}
Some properties of the generalized inverse are listed in the Appendix.

If $W \in \mc W$ then $w$ is an absolutely continuous function and thus it is the primitive of a non-negative function $a = w'$, see \ref{w_is_abs_cont} in the Appendix for a proof.

In this section we show that equation \eqref{W_PDE} is equivalent to equation \eqref{a_equation}
in the sense that $v$ and $\rho$ are related by $\rho(t,y) = v(t,W(y))$.
The equation \eqref{a_equation} has to be regarded in the weak sense defined as follows.
\begin{definition}
\label{weak_a_PDE}

Denote $\langle f,g \rangle_a = \dd \int_{\bb T} a(x) f(x) g(x) dx$.

We say that a continuous function $v: \R_+ \times \bb T \to \R$ is a weak solution of \eqref{a_equation} if, for all functions $\phi \in \HHp$,
\[
	\langle v(T), \phi \rangle_a - \langle f, \phi \rangle_a = \int_0^T \langle v(t), \phi'' \rangle_{L^2} dt.
\]
Here $v(T)$ denotes the function $v(T,\cdot)$.
\end{definition}

\subsection{Equivalence of equations}
We first characterize the space $\DW$ as a set of functions composed with $W$.
\begin{proposition}
	\label{DW_parametrization}
	Let $F \in \DW$. Then there exists $G \in \HHp$ such that $F(y) = G(W(y))$ and $ a(x) \LW F( w(x) ) = G''(x)$.
\end{proposition}
\begin{proof}
	Take $F \in \DW$, which according to Definition \ref{defDW} is represented as
	\[
		F(y) = b + c W(y) + \int_{(0,y]}  \int_0^z g(r) dr \; dW(z).
	\]	
	By the substitution rule \eqref{change_of_variables} applied to the integral with respect to $W$ above,
	
	\[
		F(y) = b + c W(y) + \int_{W(0)}^{W(y)} \int_0^{w(l)} g(r) dr \,dl.
	\]
	Taking the substitution $r=w(s)$, we obtain
	\[
		F(y) = b + c W(y) + \int_{W(0)}^{W(y)} \int_0^l g(w(s))a(s) ds\, dl.
	\]
	Finally $F(y) = G(W(y))$, where
	\[
		G(x) = b + c x + \int_0^x \int_0^l g(w(s))a(s) ds\, dl
	\]
	is clearly in $\HH$ and $G''(x) = g(w(x))a(x)$.
	In order to show that $G$ is periodic it is enough to use property \eqref{f14} and the substitution rule \eqref{change_of_variables_2}.

	For the second statement, the discussion in Definition \ref{defDW} shows that $g = \LW F$.
\end{proof}
 
Keeping this equivalence in mind, we see that the following proposition is immediate.

\begin{proposition}
\label{equations_are_equivalent}
Let $v: \R_+ \times \bb T \to \R$ be a continuous weak solution of \eqref{a_equation}.
Then the function $\rho(t,y) = v(t,W(y))$ is a weak solution of \eqref{W_PDE} as in Definition \ref{weak_W_PDE}.
\end{proposition}
\begin{proof}
The statement follows from Proposition \ref{DW_parametrization} and the identity
\[\langle f,g \rangle_a = \langle f(W(.)), g(W(.)) \rangle_{L^2}\;,\]
which is a consequence of the substitution rule \eqref{change_of_variables_2}.
\end{proof}

\subsection{Equivalence of topologies}

In Section \ref{s5} we shall prove the continuity of the solution of \eqref{a_equation} with respect to the function $a = w'$ under the weak-$L^1$ topology.

Since $\int \phi a_n dx = \int \phi d w_n$, the weak-$L^1$ convergence of the sequence $a_n$ is equivalent to the vague convergence of $w_n$. Next, we state a general result relating the convergence of $w_n$ to the convergence of $W_n$.

\begin{proposition}
	Assume that $W$ is a \textit{strictly} increasing function.
	A sequence $W_n$ converges vaguely to $W$ if and only if it converges pointwise in the continuity points of $W$.
	In that case, the sequence of generalized inverses $w_n$ converges pointwise to $w$ and thus, vaguely.
\end{proposition}
The proof of the proposition above can be found in \cite[Proposition 0.1, page 5]{R}.
In consequence, under the hypothesis of Theorem \ref{mainthm}, namely, that $W_n \to W$ vaguely, we have that $a_n \to a$ weakly in $L^1$.
\section{Continuity via Fourier transform}\label{s5}

As mentioned in the Introduction, equation \eqref{a_equation} is in correspondence with the family of complex ODE's \eqref{xi_ODE}.
In this section our aim is to prove that solutions of \eqref{xi_ODE} are bounded by an $L^1(\R)$-function of the parameter $\xi$.
Then, considering $u$ as a function $u(\xi, x)$, this will allow us to take the inverse Fourier transform with respect to $\xi$.

A quick inspection to the easy case $a\equiv 1$ gives some insight on what we should expect.
The periodic solutions of equation
\[-u''(x) + i\xi u(x) = f(x)\]
can be described by taking the Fourier series with respect to $x$. The transformed equation is
\[
	n^2 u_n + i \xi u_n = f_n
\]
and the solution is 
\[
	u(x) = \sum_{n \in \Z} \frac {f_n}{n^2 + i \xi} e^{2\pi i n x}\;.
\]
Of course, $u$ can be thought as a function $u(\xi, x)$ and we need to describe its behaviour as a function of $\xi$.

First we observe that
\[
	\sum_{n \in \Z} \frac {f_n}{n^2 + i \xi} e^{2\pi i n x} = \frac {f_0}{i \xi} + \sum_{n \neq 0} \frac {f_n}{n^2 + i \xi} e^{2\pi i n x}\;,
\]
so we have a singularity at $\xi = 0$. To avoid this problem we require $f_0 = 0$ which guarantees that $u$ is bounded.
Since the equation \eqref{a_equation} is linear with respect to $v$, this requirement do not represent a restriction.
Moreover, the solution $v(t,x)$ converges to its average $f_0$ as $t \to \infty$, thus we need $f_0 = 0$ in order to expect $v$ to be an $L^2$-function of time.

By Cauchy-Schwarz, Parseval, and the inequality $2n^2|\xi|\leq n^ 4 + \xi^ 2$,
\begin{equation}
	\label{s5intro_1}
	|u(x)| \leq \|f\|_2 \sqrt{\sum_{n \neq 0} |n^2 + i \xi|^{-2}} \leq C \|f\|_2 |\xi|^{-\frac 12}\;.
\end{equation}
Thus, if we assume only that $f \in \LLp$, then we can only expect that $|u(x)| \leq C |\xi|^{-\frac 12}$, which is not in $L^2$.
In order to obtain a sufficiently rapid decay with respect to $\xi$ we may impose the condition that $f' \in L^2$, which means $(n f_n)_{n \in \Z} \in \ell^2$ and yields
\begin{equation}
	\label{s5intro_2}
	|u(x)| \leq \|f'\|_2 \sqrt{\sum_{n \neq 0} |n(n^2 + i \xi)|^{-2}} \leq C \|f'\|_2 |\xi|^{-1}\;.
\end{equation}
However, this inequality only guarantees that $u$ is bounded by an $L^2$-function of $\xi$ and does not imply continuity in time of $v = \F^{-1}[u]$.
We observe that this bound can not be improved because the function $v$ has a discontinuity at $t = 0$.
We overcome this difficulty in the following way: subtracting to $v(t,x)$ the function $f(x) H(t) e^{-t}$, (here $H$ is the Heaveside step function) we obtain a continuous function.
The Fourier transform in time of the difference is $k(x) = u(x) - \frac {f(x)}{1+i\xi}$ and one can easily verify that
\[
	- k''(x) + i \xi k(x) = \frac {f(x) + f''(x)}{1+i\xi}\;. \]
Finally, the estimate \eqref{s5intro_1} applied to $\frac {f+f''}{1+i\xi}$ gives
\[|k(x)| \leq C \|f + f''\|_2 |\xi|^{- \frac 32}\;,\]
which yields the desired $L^1$ bound for $k$, provided by the fact that $k_0 = f_0 = 0$.
Consequently, the inverse Fourier transform 
\[v(t,x) = \mc F^{-1}(k) (t,x) + f(x) H(t) e^{-t}\]
 is continuous in time, for $t \in (0, \infty)$.

Unfortunately, the estimate \eqref{s5intro_2} fails in the general case when $a$ is not constant.
Under additional assumptions on $W$, this estimate holds and one can obtain solutions with improved regularity. This is left for future work.

\subsection{Decay speed of the solutions of \eqref{xi_ODE}}

For the rest of the section, all integrals are with respect to the Lebesgue measure.

\begin{definition}
	We will denote 
	\[\A = \left\{a\in \Lp \; ; \;a\geq 0 \hbox{ and } \int a = 1\right\}\;.\]
	For $\xi \in \bb R$ and $a \in \A$, we define the operator $T_{\xi,a}:\WWp \to \Lp$ by 
	\[
		T_{\xi,a}(u) = -u'' + i \xi a u.
	\]
	Here $\WWp$ is identified with the closed subspace of $\WW$, consisting of functions $u$ such that $u(0) = u(1)$ and $u'(0) = u'(1)$. 
\end{definition}

It is easy to prove that $T_{\xi,a}$ is a Fredholm operator of index $0$.
Saying that $u$ is a solution of equation \eqref{xi_ODE} is equivalent to $T_{\xi,a}(u) = af$.
\begin{definition}
	For $a \in \A$, we define the seminorms
	\[[u]_a = \sqrt{\int a|u|^2} \;\;\hbox{ and }\;\; \|u\|_\Ha = \sqrt{\|u'\|_2^2 + \int a|u|^2}\;.\]
	Observe that if $u$ is bounded, then $[u]_a \leq \|u\|_\infty$.
\end{definition}

\begin{lemma}
	\label{equivalentnorm}
	For any $a\in \A$, we have
	\[\|u\|_{H^1} \leq C \|u\|_\Ha\]
	where $C$ is independent of $a$.
	In particular $\|.\|_\Ha$ is a norm.
\end{lemma}
\begin{proof}
	Observe that $[u]_a^2 \geq \min\{|u|^2\}$ and
	\begin{equation*}
	\begin{split}
\|u\|_\infty^2 & = \max\{|u|^2\} = \left( \max\{|u|^2\} -\min\{|u|^2\} \right) + \min\{|u|^2\}\\
&\leq \int 2 \Re (u \o u') + [u]_a^2 \leq 2 \|u\|_2 \|u'\|_2 + [u]_a^2\;.
	\end{split}
	\end{equation*}
	By Young inequality,
	\[\|u\|_\infty^2 \leq \frac 12 \|u\|_2^2 + 2 \|u'\|_2^2 + [u]_a^2 \leq \frac 12 \|u\|_\infty^2 + 2 \|u\|_\Ha^2\;,\]
	from what we get $\|u\|_\infty^2 \leq 4 \|u\|_\Ha^2$  and finally
	\[\|u\|_{H^1}^2 \leq \|u'\|_2^2 + \|u\|_\infty^2 \leq 5 \|u\|_\Ha^2\;.\]
\end{proof}

\begin{lemma}
	\label{T_inf_bounded}
	Let $a \in \A$ and $\xi \in \R \setminus \{0\}$.
	Then $T_{\xi,a}:\WWp \to \Lp$ is an isomorphism of Banach spaces.
	Moreover, if $f$ is bounded and $T_{\xi, a}(u) = af$, then $u \in \Hp$ and $\|u\|_\Ha \leq C_\xi [f]_a$.
\end{lemma}
\begin{proof}
	For the first statement, since $T_{\xi, a}$ is a Fredholm operator of index $0$ it suffices to prove injectivity.
	Computing
	\begin{equation}\label{eq5.3}
	\begin{split}
	|\langle T_{\xi, a}u, u \rangle_{L^2}| & =\left| \int |u'|^2 + i\xi \int a |u|^2\right| \geq \min\{1,|\xi| \} \,\|u\|_\Ha^2\;,\\
	\end{split}
	\end{equation}
	we note that  $T_{\xi, a}(u) = 0$ implies $u = 0$.
	This proves the first statement.
	For the second, notice that by H\"older inequality
	\[\left|\int a f \o{ u }\;\right| \leq [f]_a [u]_a\,. \]
	Hence  $T_{\xi, a}(u) = af$ put together with \eqref{eq5.3} imply
	\[\min\{1,\xi\} \|u\|_\Ha^2 \leq [f]_a [u]_a \leq [f]_a \|u\|_\Ha\]
	and the result follows.
\end{proof}

\begin{lemma}
	\label{u_is_bounded}
	Let $a \in \A$, assume $\int a f = 0$ and  $u$ is the solution of \eqref{xi_ODE}.	
	Then $$\|u\|_\infty \leq 2 [f]_a\; .$$
\end{lemma}

\begin{proof}
	Multiply equation \eqref{xi_ODE} by $\o u$ and integrate in order to obtain
	\begin{equation}
		\label{u_is_bounded_1}
		 \int |u'|^2 + i \xi \int a |u|^2 = \int a f \o u\;.
	\end{equation}
	Denote $p = \int u \in \C$. Since $\int af=0$, 
	\[\left|\int a f \o u\;\right|=\left|\int a f \o{(u - p)} \right|\leq [f]_a [u-p]_a \leq [f]_a \|u-p\|_\infty \leq [f]_a \|u'\|_2\;,\]
	where we have used the Poincar\'e inequality in the last inequality of above.
	Notice that since the left-hand side integrals of \eqref{u_is_bounded_1} are real,
	\[ \|u'\|_2^2 \leq \left| \int |u'|^2 + i \xi \int a |u|^2 \right| \leq [f]_a \|u'\|_2\;.\]
	Consequently,
	\[ \|u - p\|_\infty \leq \|u'\|_2 \leq [f]_a\;.\]

	Now we need to bound $p$.
	Integrating each side of \eqref{xi_ODE} we get $\int a u = 0$.
	Denote the (real) inner product of complex numbers by $\langle z, w \rangle = \Re (z \o w)$.
	Equation $\int a u = 0$ implies $\int a(x) \langle u(x), p \rangle dx = 0$. Then
	\[|p|^2 \int a = \int a(x) \langle p, p \rangle dx = \int a(x) \langle p - u(x), p \rangle dx \leq \|u-p\|_\infty |p| \int a\;.\]
	We therefore obtain $ |p| \leq \|u-p\|_\infty$,	leading to $$\|u\|_\infty \leq |p| + \|u-p\|_\infty \leq 2 \|u-p\|_\infty \leq 2[f]_a\;.$$
\end{proof}

\begin{lemma}
	\label{u_falls_fast}
	Let $a \in \A$, assume $f$ is bounded and $u$ is the solution of \eqref{xi_ODE}.
	Then, for any $|\xi| \geq 1$,
	\[\|u\|_{H^1} \leq C |\xi|^{-\frac 12} [f]_a\;,\]
	where the constant $C$ is independent of $a$.
\end{lemma}
\begin{proof} By Lemma \ref{equivalentnorm}, it is sufficient to show that
\[\|u\|_{\Ha} \leq |\xi|^{-\frac 12} [f]_a\;.\]
	From \eqref{u_is_bounded_1} we have that
	\begin{equation}
		\label{u_falls_fast_1}
		\left|\xi \int a|u|^2 \right| \leq \left|\int |u'|^2 + i \xi \int a |u|^2\right| \leq [f]_a [u]_a\;.
	\end{equation}
	Thus,
	\[|\xi| \;[u]_a^2 \leq [f]_a [u]_a\;,\]
	which implies
	\begin{equation}\label{u_falls_fast_2}
[u]_a \leq |\xi|^{-1} [f]_a \leq |\xi|^{-\frac 12} [f]_a
\end{equation}
	because $|\xi| \geq 1$.
	On the other hand, from \eqref{u_falls_fast_1},
	\[\left|\int |u'|^2 \right| \leq [f]_a [u]_a \leq [f]_a^2 |\xi|^{-1}\;, \]
	implying
	\begin{equation}\label{u_falls_fast_3}
		\|u'\|_2 \leq |\xi|^{-\frac 12} [f]_a\;.
		\end{equation}
	Putting together \eqref{u_falls_fast_2} and \eqref{u_falls_fast_3}, we obtain that
	$ \|u\|_\Ha \leq  |\xi|^{-\frac 12}[f]_a$, 
	finishing the proof.
\end{proof}

\begin{lemma}
	\label{u_falls_faster}
	Let $a \in \A$, $|\xi| \geq 1$, $f \in \HHp$ and in addition $f'' = a g$, where $g$ is a bounded function.
	Let $u$ be the solution of \eqref{xi_ODE}. Then, for $k$ defined through  
\begin{equation}\label{u_falls_faster_1}
u(x) = k(x) + \frac{f(x)}{1+i\xi}\;,
\end{equation}
 holds the estimate
	\begin{equation}
	\label{u_falls_faster_2}
	|k\|_{H^1} \leq C |\xi|^{-\frac 32} [f+g]_a\;,
	\end{equation}
	where the constant $C$ does not depend on $a$.
	Moreover, if $\int af = 0$, then $\|k\|_\infty$ is bounded by $2[f + g]_a$.
\end{lemma}

\begin{proof}
		Replacing \eqref{u_falls_faster_1} into equation \eqref{xi_ODE} we obtain
	\[
		-k'' + i \xi a k = \frac {f''}{1+i \xi} + a f - a f \frac {i \xi}{1+ i \xi}\;.
	\]
	 Since $f''= a g$, we get
	\[
		-k'' + i \xi a k = \frac {a(f+g)} {1+i \xi}\;.
	\]
	Applying Lemma \ref{u_falls_fast}, we obtain
	\[
		\|k\|_{H^1} \leq C |\xi|^{-\frac 12} |1+i\xi|^{-1} [f + g]_a \leq C |\xi|^{-\frac 32} [f+g]_a\;.
	\]
	and \eqref{u_falls_faster_2} follows.
	Additionally, if we have $\dd\int a f = 0$, then $\dd\int a(f+g) = 0$ because $f$ is periodic.
	This permits to invoke Lemma \ref{u_is_bounded} proving that $\|k\|_\infty \leq 2[f+g]_a$.
\end{proof}

\subsection{Continuity of the solution of \eqref{xi_ODE}}
\label{ss52}

Throughout this section we assume the conditions of Theorem \ref{mainthm}, namely,
\begin{enumerate}[(i)]
\item $W_n \to W$ vaguely, which implies that $a_n \to a$ weakly in $L^1$;

\item $f_n'' = a_n g_n$ with $g_n(x) = \mc L_{W_n} h_n(w_n(x))$ uniformly bounded, as in Proposition \ref{DW_parametrization};

\item $f_n \to f$ in $\HHp$.
\end{enumerate}

The main result of this section is contained in the propositions \ref{v_of_a_is_continuous} and \ref{v_is_the_solution} that is, the solutions $v_n$ of \eqref{a_equation} converge to $v$ in the space $C(\bb T, C^{\frac 12 -\eps}(\R_+))$ for every $\eps > 0$.

For simplicity, we shall assume without loss of generality that 
\[\int h_n(x) dx = \int a_n(x) f_n(x) dx = 0\;.\]

\begin{lemma}
	\label{u_of_a_is_w_continuous_pointwise}
	Fix $\xi \in \R \setminus \{0\}$.
	Let $u_n(x) = T^{-1}_{\xi, a_n}(a_n f_n)$ and $u(x) = T^{-1}_{\xi, a}(a f)$. In other words, $u_n$ is the solution of \eqref{xi_ODE} with $a$ replaced by $a_n$.
	Denote $k_n(x) = u_n(x) - \frac {f_n(x)}{1+i\xi}$ and analogous definition for $k(x)$.
	Then, it holds the convergence $k_n \to k$ uniformly. 
\end{lemma}
\begin{proof}
	Recall from Lemma \ref{u_falls_faster} that
	\begin{equation}\label{u_of_a_is_w_continuous_pointwise_1}
	T_{\xi,a_n}(k_n) = a_n \frac{f_n+g_n}{1+i\xi}\;.
\end{equation}
	Since $[f_n+g_n]_{a_n} \leq \|f_n + g_n\|_\infty$ is bounded,  by Lemma \ref{T_inf_bounded}, the sequence $k_n$ is $H^1$ bounded, hence precompact in $C^0$.
	
	We shall prove that $k$ is the only $C^0$-accumulation point of $k_n$.
	Let $k^* \in C^0$ be such a point and $\phi \in C^2(\bb T)$ a test function.
	Take a subsequence of $k_n$ converging to $k^*$ and call it again $k_n$. Multiplying equation \eqref{u_of_a_is_w_continuous_pointwise_1} by $\phi$ and then integrating by parts,
	\[
		-\int k_n \phi'' + i \xi \int a_n k_n \phi = \int a_n (f_n + g_n) \phi\;.
	\]
Adding and subtracting suitable terms,
	\[
		-\int\! k_n \phi'' + i \xi \int\! a_n k^* \phi + i \xi \int\! a_n (k_n - k^*) \phi = \!\int a_n(f+g) \phi +\! \int\! a_n (f_n - f) \phi + \!\int\! (f_n'' - f'')\phi.
	\]
	Since $k_n \to k^*$ in $C^0$, $\|a_n\|_{L^1} = 1$, $a_n\to a$ weakly in $L^1$,  and $f_n \to f$ in $\HHp$, we can take limits obtaining
	\[-\int k^* \phi'' + i \xi \int a k^* \phi = \int a (f+g) \phi\;.\]
	Therefore, $k^* = T^{-1}_{\xi, a}(a (f+g)) = k$.
	Then $k_n \to k$ uniformly as desired.
\end{proof}	

Next, we recall the classical definition of weighted $L^2$-spaces, see \cite{PL} for details.
\begin{definition} For $s \geq 0$, let $L^2_s(\R)$  be the set of measurable functions $f:\R \to \R$ such that $|1+i \xi|^s f(\xi)\in L^2(\R)$.
It is a Banach space with the norm $\|f\|_s = \||1+i\xi|^s f(\xi)\|_{L^2}$. 
\end{definition}
Consider the space $C(\bb T, L^2_{1-\eps}(\R))$  with the norm
\begin{equation*}
\|k\| = \sup \Big\{ \,\|k(\cdot, x)\|_{L^2_{1-\eps}}\;;\; x\in {\bb T}\Big\}\;.
\end{equation*}

\begin{lemma}
	\label{u_of_a_is_w_continuous}

	Denote $u_n(\xi, x) = T^{-1}_{\xi, a_n}(a_n f_n)(x)$ and $k_n(\xi, x) = u_n(\xi, x) - \frac {f_n(x)}{1+i\xi}$.
	Denote similarly $u(\xi, x)$ and $k(\xi, x)$.
	Then, for every $\eps > 0$, holds the convergence $k_n \to k$ in the space $C(\bb T, L^2_{1-\eps}(\R))$.
\end{lemma}

\begin{proof}
	
	Clearly, 
	\begin{equation}\label{u_of_a_is_continuous_2}
		\begin{split}
\sup_{x\in {\bb T}} \Big\{\int |k_n(\xi, x) - k(\xi, x)|^2 \;|1+&i\xi|^{2(1-\eps)} d\xi \Big\}     \\
&\leq \int \sup_{x\in {\bb T}} \left\{ |k_n(\xi, x) - k(\xi, x)|^2\;|1+i\xi|^{2(1-\eps)}\right \} d\xi \\
		\end{split}		
	\end{equation}
	and we intend to prove that the last integral goes to $0$ as $n \to \infty$.

	Since the functions $f_n$ fulfil the conditions of lemmas \ref{u_falls_faster} and \ref{u_is_bounded},
	\begin{equation} 
		\label{u_of_a_is_continuous_1}
		\sup_{x\in {\bb T}} |k_n(\xi, x)| \leq C |1+i\xi|^{-\frac 32}\;,
	\end{equation}
	where $C$ is independent of $n$ and $\xi$.
	For each fixed $\xi \neq 0$, Lemma \ref{u_of_a_is_w_continuous_pointwise} implies that 
	\[
		\sup_{x\in {\bb T}} \left\{ |k_n(\xi, x) - k(\xi, x)|^2\right \} \to 0\,,
	\]
	as $n\to\infty$. Then using the bound \eqref{u_of_a_is_continuous_1} we find that the integrand in \eqref{u_of_a_is_continuous_2} is bounded by
	\[
		C |1+i\xi|^{-3+2(1-\eps)} \leq C |1+i\xi|^{-1-2\eps}.
	\]
	Finally, we  apply Dominated Convergence to conclude that
	\[
		\int \sup_{x\in {\bb T}} \left\{ |k_n(\xi, x) - k(\xi, x)|^2|1+i\xi|^{2(1-\eps)}\right \} d\xi \to 0\,,
	\]
	as $n \to \infty$.
\end{proof}

\begin{proposition}
\label{v_of_a_is_continuous}
Following the notation of Lemma \ref{u_of_a_is_w_continuous}, denote $v_n(t,x) = \mc F^{-1}(u_n)(t,x)$, where $\mc F$ is the Fourier transform with respect to $\xi$,  and analogous notation for  $v(t,x)$. Then, the convergence $v_n \to v$ holds in $C(\bb T,C^{\frac 12-\eps}(\R_+) )$.
\end{proposition}
\begin{proof}
	The operator $\mc F^{-1}:L^2_{1 - \eps}(\R) \to C^{\frac 12-\eps}(\R)$ is continuous \cite[Thm 3.2, pag. 47]{PL} so clearly it extends continuously to an operator $\mc F^{-1}:C(\bb T,L^2_{1-\eps}(\R)) \to C(\bb T,C^{\frac 12-\eps}(\R))$.
In consequence,   $\mc F^{-1}k_n \to \mc F^{-1}k$ in $C(\bb T,C^{\frac 12-\eps}(\R))$, as $n\to\infty$.

On the other hand, $u_n(t,x) = k_n(t,x) + \frac{f_n(x)}{1+i\xi}$ and $\mc F^{-1}\left(\frac{f_n(x)}{1+i\xi}\right) = f_n(x) H(t)e^{-t}$, where $H(t)$ is the Heaveside step function, which is not in $C^{\frac 12-\eps}(\R)$ but it is in $C^{\frac 12-\eps}(\R_+)$.

Lastly, since $f_n \to f$ uniformly, it follows that
$$f_n(x) H(t)e^{-t}\stackrel{n\to\infty}{\longrightarrow}f(x) H(t)e^{-t} $$
 in $C(\bb T,C^{\frac 12-\eps}(\R_+))$, and then
$v_n \to v$ in $C(\bb T,C^{\frac 12-\eps}(\R_+) )$, as desired.
\end{proof}

It remains only to show that  $v$ is indeed a solution of the corresponding equation.
\begin{proposition}
\label{v_is_the_solution}
	The function $v(t,x)$ from Proposition \ref{v_of_a_is_continuous} solves the equation \eqref{a_equation} in the weak sense of Definition \ref{weak_a_PDE}.
\end{proposition}

\begin{proof}
	We take the Fourier transform of $v$ with respect to time, which is the function $u(\xi, x)$ from the previous lemmas.
	We know that $\F v(\xi, x)$ solves the equation
	\[(\F v)_{xx}(\xi, x) = a(x) (i \xi \F v(\xi, x) - f(x))\;.\]
	Since we do not know \emph{a priori} if $v_{xx}$ exists, we multiply the equation above by a test function $\phi \in \CCp$ and then integrate, obtaining
	\[\int \F v(\xi, x) \phi''(x) dx = \int a(x) \phi(x) (i \xi \F v(\xi, x) - f(x)) dx\;.\]

	By Fubini Theorem,
	\[\F \left(\int v(t, x) \phi''(x) dx\right)(\xi) = i \xi \F \left( \int a(x) v(t,x) \phi(x)  dx \right)(\xi) - \int a(x) f(x) \phi(x) dx\;,\]
	which   can be written in the form
	\[\F (U)(\xi) = i \xi \F (V)(\xi) - K\,,\]
	where $U(t) = \langle v(t,.), \phi'' \rangle_{L^2}$ and $V(t) = \langle v(t,.), \phi \rangle_a$ are continuous functions, and $K = \langle  f, \phi \rangle_a$ is a constant.
	The identity above implies that $U(t)$ is the weak derivative of $V(t)$ and that $V(0) = K$.
	Finally, we infer that  $V(T) - K = \int_0^T U(t) dt$ and the statement follows.

	



\end{proof}

\begin{proof}[Proof of Theorem \ref{mainthm}]
Let $v(t,x)$ be the function defined in Proposition \ref{v_of_a_is_continuous}, which by Proposition \ref{v_is_the_solution} is the weak solution of equation \eqref{a_equation}.
Defining $\rho(t,y) = v(t,W(y))$ and recalling Proposition \ref{equations_are_equivalent}, we conclude that $\rho$ is a weak solution of \eqref{W_PDE} as in Definition \ref{weak_W_PDE}.
Finally, Proposition \ref{v_of_a_is_continuous} gives the desired convergence, concluding the proof.
\end{proof}

\section*{Acknowledgements}

\appendix
\section{Auxiliary results}
Let $w$ be the generalized inverse of $W$ defined by
\begin{equation*}
w(s):=\sup\{r\,;\,W(r)\leq s\}\;.
\end{equation*}
which is an right inverse of $W$, or else $w( W(x))=x$. If $w(x)$ is a continuity point of $W$, it holds also that
$W(w (x))=x$.

For a detailed account of properties of the generalized inverse we refer to \cite{EH} and the book \cite{R}.

\begin{proposition}
(Changing of variables).\\
For any measurable bounded function $h:[0,1] \to \bb R$,
\begin{equation}
	\label{change_of_variables}
	\int_{(0,t]}h(y)\,dW(y)\;=\;\int_{(W(0),W(t)]}h(w(x))\,dx
\end{equation}
and
\begin{equation}
	\label{change_of_variables_2}
	\int a(x) h(x) dx = \int h(W(y)) dy\;.
\end{equation}
\end{proposition}
For a proof we refer the reader to  \cite[Prop. 1, Prop. 2, page 3]{FT}.

\begin{proposition}
	\label{pullback_measure}
	Let $\mu$ be the measure associated to $W$. Then for any borel set $A \subseteq [0,1]$, $\mu(A) = \mc L(w^{-1}(A))$.
\end{proposition}
This is a known result. See, for example, \cite{R} or \cite{EH}.
	
\begin{proposition}
	\label{w_is_abs_cont}
	If $\mc L \ll \mu$, then $w$ is an absolutely continuous function.
\end{proposition}
\begin{proof}
	Since $w$ is continuous and non-decreasing, we only need to check it satisfies the Lusin property, namely that $w$ maps sets of measure zero into sets of measure zero.

	Let $E \subset [0,1]$ be a measurable set with $\mc L(E) = 0$.
	Let $\bb D(W)$ denote the set of discontinuity points of $W$.
	Clearly $\mc L(w(E)) = \mc L(w(E) \setminus \bb D(W))$ because $\bb D(W)$ is only denumerable.
		Since $\mc L \ll \mu$, by Proposition \ref{pullback_measure}, we only need to show that $\mc L(w^{-1}(w(E) \setminus \bb D(W))) = 0$.

        We know that $w(x) \not\in \bb D(W)$ implies $W(w(x)) = x$. From this fact, it easily follows  that $w^{-1}(w(E) \setminus \bb D(W)) \subseteq E$.
	Consequently, 
        \[
		\mu(w(E) \setminus \bb D(W) ) = \mc L(w^{-1}(w(E) \setminus \bb D(W))) \leq \mc L(E) = 0\;,
	\]
	as we wanted.
\end{proof}

\begin{remark}
\label{weaker_condition}
It is almost immediate that the converse of Proposition \ref{w_is_abs_cont} is also true.
In \cite{fl}, the condition $\mc L \ll \mu$ is replaced by a weaker one, namely that $\mu(I) > 0$ for every open interval $I \subset \bb T$. 
This is not enough to guarantee that $w$ is an absolutely continuous function. For instance, 
the measure $\mu = \sum_{n = 1}^\infty 2^{-n}\delta_{q_n}$, which has a delta at each rational number $q_n$ is a counterexample because it assigns positive measure to every open interval, but $w$ is a Cantor-like staircase function, thus not an absolutely continuous function.
\end{remark}

\end{document}